\documentclass[a4paper,reqno,11pt]{amsart}
\usepackage{amsmath,amssymb,amsthm,amsxtra}
\usepackage{latexsym}
\usepackage{amsbsy}
\usepackage{amscd}
\usepackage{mathrsfs}
\usepackage{tipa}
\usepackage{amsfonts}
\usepackage{graphicx}
\usepackage{listings}
\usepackage{url}
\usepackage{bookmark}
\usepackage{tikz}
\usepackage{color}%change the color of text

\numberwithin{equation}{section}

\newtheorem{theorem}{Theorem}[section]

\newtheorem{lemma}[theorem]{Lemma}
\newtheorem{corollary}[theorem]{Corollary}
\newtheorem{proposition}[theorem]{Proposition}
\newtheorem{definition}[theorem]{Definition}
\newtheorem{remark}[theorem]{Remark}
\newtheorem{example}[theorem]{Example}

\def\C{\mathcal{C}}

\def\I{\mathcal{I}}

\def\T{\mathcal{T}}
\def\F{\mathcal{F}}

\def\et{\mathbb{E}}

\def\X{\mathcal{X}}

\def\y{\mathsf{Y}}
\def\x{\mathsf{X}}
\def\t{\mathsf{T}}
\def\s{\mathsf{S}}
\def\m{\mathbf{M}}
\def\n{\mathbf{N}}
\def\uu{\mathbf{U}}
\def\vv{\mathbf{V}}
\def\ee{\mathbf{E}}
\def\c{\mathsf{C}}
\def\alg{A}
\def\salg{B}
\def\pp{\mathbf{P}}
\def\qq{\mathbf{Q}}

\def\ch{\s}

\providecommand{\Cone}{\mathop{\rm Cone}\nolimits}%
\providecommand{\Cocone}{\mathop{\rm Cocone}\nolimits}%
\providecommand{\add}{\mathop{\rm add}\nolimits}%
\providecommand{\Gen}{\mathop{\rm Gen}\nolimits}%
\providecommand{\End}{\mathop{\rm End}\nolimits}%
\providecommand{\Ext}{\mathop{\rm Ext}\nolimits}%
\providecommand{\Hom}{\mathop{\rm Hom}\nolimits}%
\providecommand{\thick}{\mathop{\rm thick}\nolimits}%
\providecommand{\per}{\mathop{\rm per}\nolimits}%
\renewcommand{\mod}{\mathop{\rm mod}\nolimits}%
\providecommand{\Ab}{\mathop{\rm Ab}\nolimits}%
\providecommand{\op}{\mathop{\rm op}\nolimits}%

\newcommand{\heart}{\ensuremath\heartsuit}
\newcommand{\butt}{\rotatebox[origin=c]{180}{\heart}}

\title{Weak cotorsion, $\tau$-tilting and two-term categories}
\dedicatory{Dedicated to our teacher and friend, Idun Reiten, on the occasion of her 80th birthday}
\author{Aslak Bakke Buan}
\address{Institutt for matematiske fag, NTNU,
 N-7491 Trondheim, Norway.}
\email{aslak.buan@ntnu.no}

\author{Yu Zhou}
\address{Yau Mathematical Sciences Center, Tsinghua University, 100084 Beijing, China}
\email{yuzhoumath@gmail.com}

\begin{document}
\begin{abstract}
Motivated by its links to $\tau$-tilting theory, we introduce a generalization of cotorsion pairs in module categories. Such pairs are also linked to co-t-structures in
corresponding triangulated categories, and to cotorsion pairs in certain extension closed
(and hence extriangulated)
subcategories, which we call two-term categories.
\end{abstract}

\thanks{This work was supported by grant number FRINAT 301375 from the Norwegian Research Council and by grant number 12031007 from National Natural Science Foundation of China. The first named author would like to thank the Isaac Newton Institute for Mathematical Sciences, Cambridge, for support and hospitality during the programme Cluster Algebras and Representation Theory where work on this paper was undertaken. This work was supported by EPSRC grant no EP/K032208/1.}

\maketitle	
	
\section*{Introduction and main results}

We introduce the notion of {\em weak cotorsion pairs} in module categories, and show that 
support $\tau$-tilting modules, as defined by Adachi-Iyama-Reiten \cite{AIR},  
give rise to such pairs. In particular, this allows us to generalize a theorem of \cite{BBOS}, 
from a tilting to a support $\tau$-tilting setting, which can be considered as a characterization of the (left) weak
cotorsion pairs which come from $\tau$-tilting.

Our main result concerns modules over a finite dimensional algebra $A$,
but we will need to study cotorsion in two-term categories,
as was recently introduced and studied by Paukzstello and Zvonareva \cite{PZ},
in order to prove our main results. 
The motivating, and for us also the most central, example of a two-term category,
is the category whose objects are maps between finite-dimensional projective $A$-modules,
and whose morphisms are chain maps up to homotopy.
Note, that this is an extension closed subcategory of a triangulated category, the homotopy category of complexes of projectives. Hence it is in particular equipped with a natural structure of an extriangulated category \cite{NP}.

In \cite{PZ}, they prove a version of the HRS-tilting theorem \cite{HRS} in the setting of
complete cotorsion pairs in two-term categories. Moreover, they also prove a correspondence between 
such complete cotorsion pairs and  functorially finite torsion pairs in a corresponding module category.
Both these results are crucial for proving our main theorems. 
We provide independent proofs of these results, adapted to our setting.
In addition, we give a version of the Brenner-Butler theorem \cite{BB} for two-term categories, which will also be needed for our main results.

In order to state precisely our main results, we need some notation.
Let $\alg$ be a finite dimensional algebra,
and let $\mod \alg$ be the category of finite dimensional left $\alg$-modules. 
Let $\Gen T$ denote the full subcategory of modules which are generated by direct sums of
copies of a module $T$ in $\mod \alg$. For a full subcategory $\X$, let $\X^{\perp} = \{Y \mid \Hom(\X,Y) = 0\}$,
and for a module $X$, let $X^{\perp}  = \{X\}^{\perp}$. 

Now recall the following important theorem from \cite{AIR}.

\begin{theorem}\label{adachi}
The map $T \mapsto (\Gen T, T^{\perp})$ gives a bijection between 
support $\tau$-tilting modules and functorially finite torsion pairs in $\mod \alg$.
\end{theorem}	

Our aim is to give a cotorsion version of this theorem, for this  
we introduce the notion of {\em left weak cotorsion pair}:

\begin{definition}\label{def:lw-cot}
	A pair $(\C,\T)$ of subcategories of $\mod \alg$ is called a left weak cotorsion pair (or lw-cotorsion pair for short) if
	\begin{enumerate}
		\item $\Ext^1(\C,\T)=0$;
		\item for any $M\in\mod \alg$, there are exact sequences
		$$0\to Y_M\to X_M\xrightarrow{f_M} M\to 0$$
		and
		$$M\xrightarrow{g^M} Y^M\to X^M\to 0$$
		with $X_M,X^M\in\C$, $Y_M,Y^M\in\T$, $f_M$ a right $\C$-approximation of $M$, and $g^M$ a left $\T$-approximation of $M$. 
	\end{enumerate}
\end{definition}

Let $^{\bot_1}{\X} = \{Y \mid  \Ext^1(Y,\X) = 0\}$ for a subcategory  $\X$.
Our first main theorem is the following.
\begin{theorem}
Let $T$ be a support $\tau$-tilting $\alg$-module. 
Then $(^{\bot_1}{\Gen T}, \Gen T)$ is a lw-cotorsion pair
in $\mod \alg$.
\end{theorem}	

Note that there are in general lw-cotorsion pairs $(\C, \T)$, such that $\T$ is {\em not}
closed under factors. Indeed, for any algebra $\alg$, it is easily verified that $(\mod \alg, \I(\alg))$,
where $\I(\alg)$ denotes the category of all injective objects in $\mod \alg$, is
a lw-cotorsion pair (in fact it is a cotorsion pair). And for any non-hereditary algebra $\alg$,
we have that $\I(\alg)$ is not closed under factors.

A triple of subcategories $(\C, \T, \F)$ is called a lw-cotorsion-torsion triple if
$(\C, \T)$ is a lw-cotorsion pair, and $(\T, \F)$ is a torsion pair. We prove the following strengthening of the above theorem, where the tilting part is \cite[Theorem 2.29]{BBOS}. Here $\add T$ denotes the additive closure of a module $T$.

\begin{theorem}
The map $T \mapsto (^{\bot_1}{\Gen T}, \Gen T, T^{\bot})$ is a bijection between basic support $\tau$-tilting modules and
lw-cotorsion-torsion triples, with inverse $(\C, \T, \F) \mapsto T$, where $\add T = \C \cap \T$.

The map specializes to a bijection between tilting modules and cotorsion-torsion triples.
\end{theorem}	

Note that a support $\tau$-tilting module $T$ is {\em tilting} if and only if it is faithful (i.e. there is a monomorphism $\alg \to T^t$ for some positive integer $t$ (see  \cite[Proposition 2.2]{AIR}), and that a lw-cotorsion pair is a {\em (complete) cotorsion pair} if the map $g^M$ in Definition \ref{def:lw-cot} can always be chosen to be injective. 

%\begin{corollary}
%	The map $T \mapsto (^{\bot_1}{\Gen T}, \Gen T)$
%	is a bijection between support $\tau$-tilting modules and functorially finite lw-cotorsion %pairs, with inverse $(\C, \T) \mapsto T$, where $\add T = \C \cap \T$.
	
%	The map specializes to a bijection between tilting modules and functorially finite cotorsion pairs.
%\end{corollary}
As a consequence, we prove the following, which also generalizes similar results of \cite{BBOS}
for tilting objects.

\begin{corollary}
For a lw-cotorsion-torsion triple $(\C, \T, \F)$, we have $$\C/(\C\cap\T)\simeq\F$$.
\end{corollary}

 By the above, also the following holds.

\begin{corollary}
For a support $\tau$-tilting module $T$, we have $$^{\bot_1}{\Gen T}/\add T \simeq T^{\perp}.$$
\end{corollary}

To prove our main results we need to consider cotorsion theory in two-term categories. We do this in Section 1.
In Section 2 we prove a cotorsion version of the tilting theorem of Brenner-Butler for two-term categories.
Then we show the relation to torsion pairs in the module category in Section 3, before we prove our main results in Section 4.

\section{Cotorsion in two-term categories and HRS-tilting}

The main results in this section are due to Pauksztello and Zvonareva \cite{PZ}.
Our proofs are independent and follow a slightly different path, and for completeness and readability we provide proofs. 

All subcategories are assumed to be full and closed under isomorphisms,
and $\alg$ is always a finite dimensional algebra.

We will consider {\em two-term categories}.
In later sections, we will only need the following special case. Let $\per \alg = K^b(\alg)$ denote the homotopy category of complexes of projective objects in $\mod \alg$.
The two-term category $\butt(\alg)$ is the full extension-closed subcategory 
of $\per \alg$ with complexes concentrated in degree 0 and -1.
Being an extension-closed subcategory of a triangulated category, $\butt(\alg)$ has
in particular the structure of an {\em extriangulated category} \cite{NP}, or $\et$-category for short. We do not give the full definition of such categories here, but refer to
\cite{NP} for detailed definitions and notation. We do however note that the bifunctor
$\et (-,-) \colon \ee^{\op} \times \ee \to \Ab$ involved in the definition of an extriangulated structure on $\ee$ is just given by $\et(X,Y) \colon = \Hom_{\t}(X,Y[1])$ in the case we consider; namely when $\ee$ is an extension-closed subcategory of 
a triangulated category $\t$. In this case we also have that the
$\et$-triangles are just the usual triangles in $\t$ whose elements belong to $\ee$.
In this section, we denote by $\t$ a triangulated category with split idempotents.

Silting theory gives rise to a more general notion of two-term categories. 

\begin{definition}[\cite{KV,AI}]
	A subcategory $\s$ of $\t$ is called a silting subcategory if the following hold:
	\begin{enumerate}
		\item $\s$ is closed under direct summands;
		\item $\Hom_\t(S_1,S_2[i])=0$ for any $S_1,S_2\in\s$ and any $i>0$;
		\item $\thick(\s)=\t$.
	\end{enumerate}
\end{definition}

\begin{remark}
    By \cite[Theorem~2.9]{IY}, if a triangulated category has an idempotent complete silting subcategory, then it is idempotent complete.
\end{remark}

If $\s$ is the additive closure $\add S$ of an object $S$, we say that $S$ is a silting object,
and sometimes just write $S$ for $\add S$. We note that, in particular, $\alg$ is a silting object in $\per \alg$.

For subcategories $\x, \y$ of $\t$, we let $\x \ast \y$ denote the subcategory
with objects $Z$ which occur in triangles
$$X \to Z \to Y \to X[1]$$ with $X \in \x$ and $Y \in \y$.
We note that with this notation, we have  $\butt(\alg) = \alg \ast \alg[1]$.

For any silting category $\s$ in $\t$, we can also consider the
{\em two-term category} $\s \ast \s[1]$, and we will see that is an $\et$-category.

We proceed by recalling the notion of cotorsion pairs in an $\et$-category, from \cite{LN}.

\begin{definition}
	A pair of subcategories $(\uu,\vv)$ of an $\et$-category $\ee$ is called a cotorsion pair if for any $X\in\ee$, we have
	\begin{enumerate}
		\item $\et(\uu,X)=0$ if and only if $X\in\vv$;
		\item $\et(X,\vv)=0$ if and only if $X\in\uu$.
	\end{enumerate}
\end{definition}

The following follows directly from the definition of cotorsion pairs.

\begin{lemma}\label{lem:ex}
	Let $(\uu,\vv)$ be a cotorsion pair in an $\et$-category $\ee$. Then $\uu$ and $\vv$ are closed under extensions and direct summands.
\end{lemma}

For two subcategories $\uu$ and $\vv$ of an $\et$-category $\ee$, define the following  subcategories of $\ee$:
\[\Cone(\vv,\uu)=\{X\in\ee\mid \exists\ \text{an $\et$-triangle}\ V\to U\to X\ \text{with}\ U\in\uu, V\in\vv \},\]
and
\[\Cocone(\vv,\uu)=\{X\in\ee\mid \exists\ \text{an $\et$-triangle}\ X\to V\to U \text{with}\ U\in\uu, V\in\vv \}.\]

\begin{definition}
	A cotorsion pair $(\uu,\vv)$ in $\ee$ is called complete if $$\ee=\Cone(\vv,\uu)=\Cocone(\vv,\uu).$$
\end{definition}

\begin{lemma}\label{lem:complete}
	A pair $(\uu,\vv)$ of subcategories of the extriangulated category $\ee$ is a complete cotorsion pair if and only if
	\begin{itemize}
		\item[-] $\uu$ and $\vv$ are closed under direct summands;
		\item[-]  $\et(\uu,\vv) =0$ and
		\item[-]  $\ee =\Cone(\vv,\uu)=\Cocone(\vv,\uu)$.
	\end{itemize}
\end{lemma}

\begin{proof}
    The ``only if'' follows from the definition of cotorsion pairs and Lemma~\ref{lem:ex}. So we only need to show the ``if" part. For any $X\in\ee$, there is an $\et$-triangle $$V\to U\to X$$ with $V\in\vv$ and $U\in\uu$. If $X$ satisfies that $\et(X,\vv)=0$, then $X$ is a direct summand of $U$ and hence is in $\uu$. 
\end{proof}

\begin{remark}\label{exm:tri}
	If we consider the triangulated category $\t$ as an extriangulated category
	in the natural way, with $\et(X,Y) \colon = \Hom_{\t}(X,Y[1])$,
	then a cotorsion pair in $\t$ is a pair of subcategories $(\uu,\vv)$ satisfying
	\begin{itemize}
		\item[-]  $\uu$ and $\vv$ are closed under direct summands;
		\item[-]  $\Hom(\uu,\vv[1])=0$; and
		\item[-]  $\t=\uu\ast\vv[1]$.
	\end{itemize}
	This is because $\Cone(\vv,\uu)=\uu\ast\vv[1]$ and $\Cocone(\vv,\uu)=\uu[-1]\ast\vv$. So $(\uu,\vv)$ is a cotorsion pair in a triangulated category if and only if $(\uu,\vv[1])$ is a torsion pair in the sense of \cite{IYo} if and only if $(\uu,\vv)$ is a cotorsion pair in the sense of \cite{N}.
\end{remark}

Silting theory is closely connected to the theory of co-t-structures, as introduced by Paukzstello \cite{P} and Bondarko \cite{B}.
 
\begin{definition}
	A co-$t$-structure on the triangulated category $\t$ is a pair of subcategories $(\x, \y)$ satisfying
	\begin{enumerate}
		\item $\x$ and $\y$ are closed under direct summands;
		\item $\x[-1]\subset \x$;
		\item $\Hom(\x,\y[1])=0$;
		\item $\t=\x\ast \y[1]$.
	\end{enumerate}
\end{definition}

Comparing with the notion of cotorsion pair, we have the following.

\begin{remark}\label{rmk:0}
	A pair of subcategories $\c=(\x,\y)$ of $\t$ is a co-$t$-structure if and only if it is a complete cotorsion pair in $\t$ such that $\x[-1]\subset \x$. 	This follows form Remark~\ref{exm:tri}.
%	%The co-heart of a co-$t$-structure is exactly its core as a cotorsion pair. The heart of a %co-$t$-structure as a cotorsion pair is $(Cone(\y,\ch_\K)\cap %CoCone(\ch_\K,\x))/\ch_\K\subset(\y\cap\x)/\ch_\K=0$.
\end{remark}
%´

\begin{definition}
	A co-$t$-structure $(\x,\y)$ is called \emph{bounded} if
	$$\t=\bigcup_{n\in\mathbb{Z}}\x[n]=\bigcup_{n\in\mathbb{Z}}\y[n].$$
\end{definition}

%\begin{definition}[\cite{KV,AI}]
%	A subcategory $\s$ of $\t$ is called a silting subcategory if the following hold:
%	\begin{enumerate}
%		\item $\s$ is closed under direct summands;
%		\item $\Hom_\C(S_1,S_2[i])=0$ for any $S_1,S_2\in\s$;
%		\item $\thick(\s)=\t$.
%	\end{enumerate}
%\end{definition}

For a silting subcategory $\s$ of $\t$, we define subcategories of $\t$:

$$\x_\s= \ \bigcup_{n\leq 0}\s[n]\ast\cdots\ast\s[-1]\ast\s$$
and
$$\y_\s= \ \bigcup_{n\geq 0}\s\ast\s[1]\cdots\ast\s[n].$$
%By definition, we have $\K^{\leq m}(\s)=\K^{\leq 0}(\s)[m]$ and $\K^{\geq m}(\s)=\K^{\geq 0}(\s)[m]$.

\begin{definition}
	Let $\c = (\x,\y)$ be a co-$t$-structure pair in a triangulated  category $\t$. The subcategory $\ch_{\c} = \x \cap \y$ is called the co-heart of $\c$, and the subcategory $\butt(\c)  =\x[1 ]\cap\y$ is called the {\em extended co-heart} of $\c$.
\end{definition}

Co-hearts of bounded 
co-$t$-structures are silting subcategories, in fact we have the following bijection.

\begin{proposition}[{\cite[Corollary~5.9]{MSSS}}, cf. also {\cite[Proposition~2.8]{IY}}]\label{prop:msss}
	The map 
	$$\c = (\x,\y)\mapsto \ch_\c$$
	is a bijection between bounded co-$t$-structure and silting subcategories of $\t$, with inverse
	$$\s\mapsto(\x_\s,\y_\s).$$
	%where
	%$$\K^{\leq 0}(\s)=\cup_{n\geq 0}\s[-n]\ast\cdots\ast\S[-1]\ast\S$$
	%and
	%$$\K^{\geq 0}(\S)=\cup_{n\geq 0}\S\ast\S[1]\cdots\ast\S[n].$$
\end{proposition}

%For a silting subcategory $\s$, we denote $$\butt(\s)=\butt((\x_\s,\y_\s)).$$

The extended co-heart $\butt(\c)$ of a co-$t$-structure $\c$ is closed under extensions and direct summands. In particular, $\butt(\c)$ is an $\et$-category. In fact, by results of Iyama, J{\o}gensen and Yang, the extended co-hearts are exactly "two-term categories", in the following sense.

\begin{lemma}[{\cite[Lemma~2.1]{IJY}}]\label{lem:ijy}
		Let $\c=(\x,\y)$ be a bounded co-$t$-structure in a triangulated category $\t$.
 Then
	$$\butt(\c)=\ch_\c\ast\ch_\c[1]=:\butt(\ch_\c).$$
\end{lemma}
%
%\begin{proof}
%	This follows from Proposition~\ref{prop:msss}.
%\end{proof}

\begin{proposition}[{\cite[Theorem~2.3]{IJY}}]\label{prop:ijy}
	The bijection in Proposition~\ref{prop:msss} restricts to a bijection between the set of bounded co-$t$-structures $\c'= (\x',\y')$ with $\x\subset\x' \subset\x[1]$ and the set of silting subcategories of $\t$ which are in $\butt(\c)$.
\end{proposition}

For the next lemmas, fix a  a bounded co-$t$-structure $\c=(\x,\y)$  in a triangulated category $\t$. 

\begin{lemma}\label{lem:sub}
	Let $\s$ be a silting subcategory of $\t$, with $\s \subseteq \butt(\c)$. Then $\x \subset\x_\s \subset\x[1]$ and $\y[1]\subset\y_\s \subset\y$.
\end{lemma}

\begin{proof}
	This is a direct consequence of Proposition \ref{prop:ijy}.
	\end{proof}

\begin{lemma}\label{lem:cap}
	Let $\s$ be a silting subcategory of $\t$, with $\s \subseteq \butt(\c)$.  
	Then $$(\x_\s\cap\butt(\c),\y_\s\cap\butt(\c))$$ is a complete cotorsion pair in $\butt(\c)$.
\end{lemma}

\begin{proof}
	First, since $\x_\s$, $\y_\s$ and $\butt(\c)$ are closed under direct summands, so are $\x_\s\cap\butt(\c)$ and $\y_\s\cap\butt(\c)$.
	
	Second, note that the $\et$-structure of $\butt(\c)$ is inherited from the triangulated structure of $\C$. So $\et(\x_\s\cap\butt(\c),\y_\s\cap\butt(\c))=0$.
	
	%for any $X\in K^{\leq0}(\S)\cap\butt(\P)$, and any	 $Y\in K^{\geq 0}(\S)\cap\butt(\P)$, we have $\Ext^1(X,Y)=0$. 
	
	Finally, by Remark~\ref{rmk:0}, $(\x_\s,\y_\x)$ is a complete cotorsion pair in $\t$. Then we have $\t=\Cone(\y_\s,\x_\s)=\Cocone(\y_\s,\x_\s)$. So for any $Z\in\butt(\c)$, on one hand, there is a triangle
	\[ Y\to X\to Z\to Y[1] \]
	with $X\in \x_\s$ and $Y\in \y_\s$. 
	%Since $\s\subset\ch_\K\ast\ch_\K[1]$, 
	By Lemma~\ref{lem:sub}, we have $X\in \x[1]$. 
	% for some $n\geq 0$.
	Then $$Y\in\butt(\c)[-1]\ast\x[1]\subset \x \ast \x[1]\subset \x[1] \ast \x[1] = \x[1].$$ 
	But on the other hand, since $Y\in \y_\s$, by Lemma~\ref{lem:sub}, we also have $Y\in\y$. Hence $Y\in\butt(\c) = \x[1] \cap \y$. Since $Y,Z \in \butt(\c)$, so is $X$. Thus, we get an $\et$-triangle
	$$Y\to X\to Z$$
	in $\butt(\c)$, with $X\in\x_\s\cap\butt(\c)$, and $Y\in\y_\s\cap\butt(\c)$. This implies $$\butt(\c)=\Cone(\y_\s\cap\butt(\c),\x_\s\cap\butt(\c)).$$
	
	On the other hand, we have a triangle
	\[Z\to Y'\to X'\to Z[1] \]
	with $X'\in \x_\s$ and $Y'\in \y_\s$. Then by Lemma~\ref{lem:sub}, we have
	$$X'\in \y_\s\ast\butt(\c)[1]\subset \y\ast\butt(\c)[1] \subset \y \ast \y[1] \subset \y \ast \y = 
	\y.$$
	By Lemma~\ref{lem:sub}, we also have $X'\in\x_S\subset \x[1]$. So $X'\in\butt(\c) = \x[1] \cap \y$. Then so is $Y'$. Thus, we get an $\et$-triangle
	$$Z\to Y'\to X'$$
	in $\butt(\c)$, with $X'\in\x_\s\cap\butt(\c)$, and $Y'\in\y_\s\cap\butt(\c)$. This implies $$\butt(\c)=\Cocone(\y_\s\cap\butt(\c),\x_\s\cap\butt(\c)).$$
	Hence by Remark~\ref{exm:tri}, the proof is complete.
\end{proof} 

The cotorsion pair from the previous Lemma has the following alternative descriptions.

\begin{lemma}\label{lem:short}
	Let $\s$ be a silting subcategory of $\t$, with $\s \subseteq \butt(\c)$. Then
	\[\x_\s\cap\y=\x_\s\cap\butt(\c)=\butt(\s)[-1]\cap\butt(\c),\]
	\[\y_\s\cap\x[1]=\y_\s\cap\butt(\c)=\butt(\s)\cap\butt(\c). \]
\end{lemma}

\begin{proof}
	By Lemma~\ref{lem:sub}, we have $\x_\s\subset\x[1]$, which implies $\x_\s\cap\y=\x_\s\cap\x[1]\cap\y=\x_\s\cap\butt(\c)$. To show $\x_\s\cap\butt(\c)=\butt(\s)[-1]\cap\butt(\c)$, since $\butt(\s)[-1]=\x_\s\cap\y_\s[-1]\subset\x_\s$, it is enough to show the inclusion $\x_\s\cap\butt(\c)\subset\butt(\s)[-1]\cap\butt(\c)$. For any $X\in\x_\s\cap\butt(\c)$, by $X\in\x_\s$, there is a triangle
	\[Y\to X\to Z\to Y[1] \]
	with $Y\in\x_\s[-2]$ and $Z\in\s[-1]\ast\s$. Then by Lemma~\ref{lem:sub}, we have $Y\in\x_\s[-2]\subset\x[-1]$. Then there is no zero map from $Y$ to $X$ by $X\in\butt(\c)\subset\y$. Hence $X$ is a direct summand of $Z$. So it is in $\butt(\s)[-1]$.
	
	The equations $\y_\s\cap\x[1]=\y_\s\cap\butt(\c)=\butt(\s)\cap\butt(\c)$ can be shown similarly.

	%\[(\S[-n]\ast\cdots\ast\S[-1]\ast\S)\cap(\ch_\K\ast\ch_\K[1])\subset(\S[-1]\ast\S)\cap(\ch_\K\ast\ch_\K[1])\]
	%for any integer $n\geq 2$. For any $X\in(\S[-n]\ast\cdots\ast\S[-1]\ast\S)\cap(\ch_\K\ast\ch_\K[1])$, there is a triangle
	%\[Y\to X\to Z\to Y[1] \]
	%with $Y\in\S[-n]\ast\cdots\S[-2]$ and $Z\in\S[-1]\ast\S$. Since $\S\subset\P\ast\P[1]$, so $Y\in\P[-n]\ast\cdots\ast\P[-1]$. Then there is no nonzero morphism from $Y$ to $X\in\P\ast\P[1]$. Thus, $X\in\S\ast\S[1]$.
\end{proof}

\begin{lemma}\label{lem:projcon}
	Let $(\uu,\vv)$ be a cotorsion pair in $\butt(\c)$. Then $\s_\c \subset\uu$ and  $\s_\c[1]\subset\vv$.
\end{lemma}

\begin{proof}
	By Lemma~\ref{lem:ijy}, we have $\butt(\c) = \s_\c \ast \s_\c[1]$. So $\Hom(\s_\c ,\butt(\c))[1])=0$ and $\Hom( \butt(\c), (\s_\c[1])[1])=0$ hold. These imply $\s_\c\subset\uu$ and $\s_\c[1]\subset\vv$, respectively, by the definition of cotorsion pairs.
\end{proof}

\begin{lemma}\label{lem:[0,2]}
	Let $(\uu,\vv)$ be a complete cotorsion pair in $\butt(\c)$. Then we have 
	\begin{itemize}
		\item[(i)] $\butt(\c)=(\uu[-1]\ast\vv)\cap(\uu\ast\vv[1])$, and
		\item[(ii)] $\uu\ast\vv[1]=\s_\c\ast\s_\c[1]\ast\s_\c[2]$.
		\end{itemize}
\end{lemma}

\begin{proof}
	By the completeness of the cotorsion pair $(\uu,\vv)$, we have $$\butt(\c)\subset(\uu[-1]\ast\vv)\cap(\uu\ast\vv[1]).$$ Conversely, $$(\uu[-1]\ast\vv)\cap(\uu\ast\vv[1])\subset(\s_\c[-1]\ast\s_\c\ast\s_\c[1])\cap(\s_\c\ast\s_\c[1]\ast\s_\c[2])\subset\butt(\c).$$ This proves (i).
	
	For (ii), the inclusion $\uu\ast\vv[1]\subset\s_\c\ast\s_\c[1]\ast\s_\c[2]$ follows direct from $\uu,\vv\subset\s_\c\ast\s_\c[1]$. For the converse inclusion, it follows from part (i) that $\butt(\c)\cap\butt(\c)[1]\subset\uu\ast\vv[1]$. So we have
	$\s_\c\ast\s_\c[1]\ast\s_\c[2]\subset\butt(\c)\ast\butt(\c)[1]\subset\uu\ast\vv[1].$
\end{proof}

\begin{proposition}[{\cite[Theorem 2.1]{PZ}}]\label{prop:cot-cotor}
	The map
	$$\c' =(\x', \y') \mapsto (\butt(\c')[-1]\cap\butt(\c),\butt(\c')\cap\butt(\c))$$
	is a bijection between bounded co-$t$-structures $\c'$ with $\x \subset\x' \subset\x[1]$ and complete cotorsion pairs in $\butt(\c)$, with inverse
	$$(\uu,\vv)\mapsto (\x[-1]\ast\uu,\vv\ast\y[2]).$$
\end{proposition}

\begin{proof}
	Let $\c'=(\x',\y')$ be a bounded co-$t$-structure. By \ref{prop:msss}, there is a silting subcategory $\s':=\s_{\c'}$ of $\t$ such that $\x'=\x_{\s'}$ and $\y'=\y_{\s'}$. Then by Lemmas~\ref{lem:cap} and \ref{lem:short}, we have that 
	$$(\butt(\c')[-1]\cap\butt(\c),\butt(\c')\cap\butt(\c))=(\x'\cap\y,\y'\cap\x[1])$$ is a complete cotorsion pair in $\butt(\c)$.

	Conversely, let $(\uu,\vv)$ be a complete cotorsion pair in $\butt(\c)$. We divide the proof of $(\x[-1]\ast\uu,\vv\ast\y[2])$ being a co-$t$-structure in $\t$ into the following three steps. 
	\begin{enumerate}
		\item First, since $\uu,\vv\subset\butt(\c)=\x[1] \cap \y$, $\Hom(\uu,\vv[1])=0$, $\x[-1]\subset\x$ and $\y[1]\subset\y$, we have
	\begin{equation}\label{eq:pf}
		\Hom(\x[\leq-1],\uu\cup\vv)=0,\ \Hom(\uu\cup\vv,\y[\geq 2])=0,\ \Hom(\x[\leq -1],\y[\geq 0])=0.
	\end{equation}
	Hence we have $$\Hom(\x[-1]\ast\uu, (\vv\ast\y[2])[1])=0.$$ 
	\item Next, we have 
	\begin{align*}
	\x[-1]\ast\uu\ast(\vv\ast\y[2])[1]) &= \x[-1]\ast\uu\ast\vv[1]\ast\y[3] \\   
	&=\x[-1]\ast\s_\c\ast\s_\c[1]\ast\s_\c[2]\ast\y[3] \\
	&=\t,
	\end{align*}
	where the second equality is due to Lemma~\ref{lem:[0,2]}, and the last equality is due to Proposition~\ref{prop:msss}. 
	\item Finally, for any summand $Z$ of an object in $\x[-1]\ast\uu$, by step (2), there is a triangle
	$$N\to M\to Z\xrightarrow{f} N[1]$$
	with $N\in\vv\ast\y[2]$ and $M\in\x[-1]\ast\uu$. By step (1), $f=0$, which implies that $M\cong N\oplus Z$. So there is a triangle
	$$X[-1]\xrightarrow{\left(\begin{smallmatrix}f_1\\f_2\end{smallmatrix}\right)}N\oplus Z\xrightarrow{\left(\begin{smallmatrix}g_1&g_2\end{smallmatrix}\right)}U\to X$$
	with $X\in\x$ and $U\in\uu$. By \eqref{eq:pf}, we have $f_1=0$. Hence there is a decomposition $U\cong U_1\oplus U_2$ such that $N\cong U_1$ and there is a triangle
	$$X[-1]\xrightarrow{f_2}Z\to U_2\to X.$$
	Note that $\uu$ is closed under direct summands in $\butt(\c)$ by Lemma~\ref{lem:complete}, and $\butt(\c)$ is closed under direct summands in $\t$. So we have $U_2\in\uu$, and hence, $Z\in\x[-1]\ast\uu$. Thus, $\x[-1]\ast\uu$ is closed under direct summands. Similarly, one can show that $\vv\ast\y[2]$ is also closed under direct summands.
	\end{enumerate}
	Thus, $(\x[-1]\ast\uu,\vv\ast\y[2])$ is a co-$t$-structure in $\t$. We claim that $\x \subset \x[-1]\ast\uu \subset \x[1]$. The second inclusion is clear. For the first inclusion, let $X \in \x$ and consider the canonical sequence 
	$$\tilde{X}[-1] \to X \to \tilde{Y} \to \tilde{X} $$
	with $\tilde{X}  \in \x$ and $\tilde{Y}  \in \y$.
	Since $\x$ is extension closed, we have $\tilde{Y}  \in \x\cap\y=\s_\c$. So by Lemma~\ref{lem:projcon}, we have $\tilde{Y}\in\uu$, which gives the second inclusion. Moreover, we have
	$$\t=\bigcup_{n\in\mathbb{Z}}\x[n]\subset\bigcup_{n\in\mathbb{Z}}(\x[-1]\ast\uu)[n].$$
	Similarly, one can show that $\t=\bigcup_{n\in\mathbb{Z}}(\vv\ast\y[2])[n]$. Hence $(\x[-1]\ast\uu,\vv\ast\y[2])$ is a bounded co-$t$-structure of $\t$.
	
	We proceed to show that the maps defined are inverse bijections. First note that we have shown that the first map can be represented as $(\x', \y') \mapsto (\x' \cap \y, \y' \cap \x[1])$. 	We therefore need to prove the following four equations:
	\begin{itemize}
	\item[(i)] $\uu = (\x[-1]\ast \uu) \cap \y$,
	\item[(ii)]  $\vv = (\vv \ast \y[2]) \cap \x[1]$,
	\item[(iii)] $\x' = \x[-1]\ast (\x' \cap \y)$,
	\item[(iv)] $\y' = (\y' \cap \x[1] )\ast \y[2]$.
	\end{itemize}
	We will only prove (i) and (iii). The proofs for (ii) and (iv) are similar.
	
	To prove (i), we note that the inclusion $\uu \supset (\x[-1]\ast \uu) \cap \y$
	follows directly from $\Hom(\x[-1], \y)= 0$ and that $\uu$ is closed under direct summands. The opposite inclusion is from $\uu\subset\butt(\c)\subset\y$.
	
	We then prove (iii). Let $Z$ be an object in $\x'$ and consider the canonical triangle 
	$$X[-1] \to Z \to Y \to X$$
	with $X \in \x \subset \x'$ and $Y \in \y$.  Then since $\x'$ is extension closed, we also have that  $Y \in \x'$, so $\x' \subset \x[-1]\ast (\x' \cap \y)$ holds. The opposite inclusion is from $\Hom(\x[-1], \y)= 0$ and that $\x'\cap\y$ is closed under direct summands, and hence (iii) holds.
\end{proof}

\begin{remark}
In \cite[Theorem~2.1]{PZ}, they take additive hulls of $\x[-1]\ast\uu$ and $\vv\ast\y[2]$. However, under the assumption that $\t$ has split idempotents, this is not necessary.
\end{remark}

Combining the bijection in Proposition~\ref{prop:cot-cotor} and the bijection in Proposition~\ref{prop:ijy}, we have the following consequence.

\begin{corollary}\label{cor:sil-cotor}
	The map
	$$\s\mapsto ((\s[-1]\ast\s)\cap\butt(\c),(\s\ast\s[1])\cap\butt(\c))$$
	is a bijection between silting subcategories of $\t$ in $\butt(\c)$ and complete cotorsion pairs in $\butt(\c)$, with inverse
	$$(\uu,\vv)\mapsto \uu\cap\vv.$$
\end{corollary}

The bijection of Proposition \ref{prop:cot-cotor} has the following direct consequence,
which we regard as a HRS-tilting theorem for two-term categories.

\begin{theorem}[{\cite[Corollary 2.4]{PZ}}]\label{thm:hrs}
	Let $(\uu,\vv)$ be a complete cotorsion pair in $\butt(\c)$. Then the intersection $\Cocone(\uu[1],\vv)\cap \Cone(\uu[1],\vv)$ is the extended co-heart of a bounded co-$t$-structure, where	
	$(\vv,\uu[1])$ is a complete cotorsion pair with $\vv\cap\uu[1]=\s_\c[1]$.
\end{theorem}

\begin{proof}
	Let $\s=\uu\cap\vv$. Then by Corollary~\ref{cor:sil-cotor}, $\s$ is silting,  $\uu=\butt(\s)[-1]\cap\butt(\c)$ and $\vv=\butt(\s)\cap\butt(\c)$. Then $\uu[1]=\butt(\c)[1]\cap\butt(\s)=\butt(\s_\c[1])\cap\butt(\s)$ and $\vv=\butt(\s_\c[1])[-1]\cap\butt(\s)$. Hence by Corollary~\ref{cor:sil-cotor} again, $(\vv,\uu[1])$ is a complete cotorsion pair in $\butt(\s)$ with $\vv\cap\uu[1]=\s_\c[1]$.
	
	By Lemma~\ref{lem:ijy}, $\butt(\s)$ is the extended co-heart of the bounded co-$t$-structure induced by $\s$. Then by Lemma~\ref{lem:[0,2]}, we have $\butt(\s)=\Cocone(\uu[1],\vv)\cap \Cone(\uu[1],\vv)$ and we are done.
\end{proof}

%\begin{definition}
%	
%\end{definition}	
%
%\begin{theorem}
%
%\end{theorem}

\section{The Brenner-Butler theorem in two-term categories}

We now return to the setting of a finite dimensional algebra $\alg$. 
Recall that the classical Brenner-Butler theorem \cite{BB} for module categories, says that
a classical tilting module $T$ in $\mod A$ with endomorphism ring $B$, gives rise to torsion pairs  
$(\T_A, \F_A)$ in $\mod A$ and $(\T_B, \F_B)$ in $\mod B$, with natural equivalences
$\T_A \simeq \F_B$ and $\T_B \simeq \F_A$. In \cite{BZ}, we gave a version of this for 2-term silting complexes,
still formulated in terms of torsion pairs in module categories.
In this section we provide an alternative approach, but now in the setting of cotorsion pairs for two-term categories.

For a two-term silting complex $\pp$, we have seen in the previous section that the subcategories of $\butt(\alg)$ given by
$$\uu(\pp)=(\add\pp[-1]\ast\add\pp)\cap\butt(\alg),\ \vv(\pp)=(\add\pp\ast\add\pp[1])\cap\butt(\alg)$$
is a complete cotorsion pair in $\butt(\alg)$. 
In the spirit of the Brenner-Butler theorem, we shall compare these to a corresponding complete cotorsion pair in the two-term category of the endomorphism ring of $\pp$.

Let $\per\alg$ be the prefect category of $A$, which is the same as the bounded homotopy category of projectives. Let $\butt(\alg):=\butt(\add\alg)$ be the two-term category sitting inside as an extension-closed subcategory.

Let $\pp$ be a two-term silting complex in $\per \alg$, i.e. $\add\pp$ is a silting subcategory of $\butt(\alg)$. 
Let $\salg=\End_{\per \alg}(\pp)$ and let $\salg^{dg}$ be the differential graded endomorphism algebra of $\pp$. Then $\salg^{dg}$ is a non-positive differential graded algebra with $H^0(\salg^{dg})=\salg$. Let $p:\salg^{dg}\to\salg$ be the canonical projection. Then its induction functor $p_\ast:\per\salg^{dg}\to\per\salg$ restricts to an additive equivalence $\add_{\per\salg^{dg}}\salg^{dg}\simeq\add_{\per\salg}\salg$. Combining $p_\ast$ with the equivalence $\per \salg^{dg}\simeq\per\alg$ sending $\salg^{dg}$ to $\pp$ given in \cite{K}, we get a triangle functor
$$\Psi\colon\per\alg\to\per\salg,$$
which restricts to an additive equivalence $\add_{\per\alg}\pp\simeq\add_{\per\salg}\salg.$

We can now give a precise formulation of the Brenner-Butler theorem for two-term categories. Note that (a) is from Corollary~\ref{cor:sil-cotor}, and (b) is \cite[Theorem~1.1~(c)]{BZ}.

\begin{theorem}\label{thm:BB}
    Let $\pp$ be a two-term silting complex in $\per \alg$ and let $\salg=\End_{\per \alg}(\pp)$.
    \begin{enumerate}
        \item[(a)] The pair $(\uu(\pp),\vv(\pp))$ is a complete cotorsion pair in $\butt(\alg)$ such that $$\uu(\pp)\cap\vv(\pp)=\add\pp.$$
        Conversely, any complete cotorsion pair in $\butt(\alg)$ arises in this way.
        \item[(b)] $\qq:=\Psi(\alg[1])$ is a two-term silting complex in $\per\alg$.
    \end{enumerate}
    Let $\m(\pp)=\Psi(\vv(\pp))$ and $\n(\pp)=\Psi(\uu(\pp)[1])$.
    \begin{enumerate}
        \item[(c)] The pair $(\m(\pp),\n(\pp))$ is a complete cotorsion pair in $\butt(\salg)$.
        \item[(d)] We have the equation 
        $$(\m(\pp),\n(\pp))=(\uu(\qq),\vv(\qq)).$$
        \item[(e)] There are equivalences $$\m(\pp)\simeq\vv(\pp)/(\pp[1],\pp),\ \n(\pp)\simeq\uu(\pp)/(\pp,\pp[-1]),$$ induced by $\Psi$, where $(X,Y)$ is the ideal consisting of morphisms which factors through $\add X$ and $\add Y$ in order.
    \end{enumerate}
\end{theorem}

\begin{remark}\label{rmk:fac}
    By Theorem~\ref{thm:BB}~(a), for $A[1]\in\butt(A)$, there is a canonical triangle
    \begin{equation}\label{eq:can}
         A\to V\to U\to A[1]
    \end{equation}
    with $U\in\uu(\pp)$ and $V\in\vv(\pp)$. By Lemma~\ref{lem:projcon}, we have $A\in\uu(\pp)$ and $A[1]\in\vv(\pp)$. Hence both $U$ and $V$ are in $\uu(\pp)\cap\vv(\pp)=\add\pp$. Thus, the triangle \eqref{eq:can} is the one in \cite[Theorem~1.1~(b)]{BZ}. Consequently, any morphism from an object in $\uu(\pp)$ to an object in $\add A[1]$ factors through $\uu(\pp)\cap\vv(\pp)$.
\end{remark}

To show this theorem, we need some preparations.

\begin{lemma}\label{lem:BY}
Let $\butt(\pp)=\butt(\add\pp)$.
\begin{enumerate}
\item $\Psi$ induces an equivalence 
$$\butt(\pp)/(\pp[1],\pp)\simeq\butt(\salg).$$
In particular, $\Psi$ induces a bijection between isoclasses of objects of $\butt(\pp)$ and those of $\butt(\salg)$.
\item For any objects $X$ and $Y$ of $\butt(\pp)$, we have an isomorphism of vector spaces
$$\Hom_{\per\alg}(X,Y[1])\cong\Hom_{\per\salg}(\Psi(X),\Psi(Y)[1]).$$
\end{enumerate}
\end{lemma}

\begin{proof}
This is essentially from \cite[Appendix A]{BY}. By \cite[Proposition~A.5]{BY}, $p_\ast$ induces an equivalence
$$\butt(\salg^{sg})/(\salg^{sg}[1],\salg^{sg})\simeq\butt(\salg).$$
So we have (1). The equivalence in (2) follows from the proof of \cite[Proposition~A.3]{BY} after replacing $X$ (resp. $X'$, $X''$) with $Y$ (resp. $Y'$, $Y''$) in the second position of each $\Hom(-,-)$ there.
\end{proof}

%\begin{lemma}
%The functor $\Psi$ in Lemma~\ref{lem:BY} induces an equivalence
%\end{lemma}

%\begin{proof}
%which restricts to an equivalence $$(\add \salg^{dg}\ast\add \salg^{dg}[1])_{\per \salg^{dg}}/(\s[1],\s)\simeq (\add \salg\ast\add \salg[1])_{K^b(\salg)}.$$ 
%\end{proof}

\begin{remark}\label{rmk:sil}
    Taking $Y=X$ in the formula in Lemma~\ref{lem:BY}~(3), one can get \cite[Proposition~A.3]{BY}, that $\Psi$ induces a bijection from isomorphism classes of silting complexes in $\butt(\pp)$ to those in $\butt(\salg)$. %Note that the silting complex $\Psi(\alg[1])$ in $\butt(\salg)$ is the one in \cite[Theorem~1.1~(c)]{BZ}.
	%Note that the equivalence in Lemma~\ref{lem:BY} induces a bijection between silting complexes in the categories on both sides. This bijection is also given directly by the map $\salg^{dg}\to \salg$. See \cite[Proposition~A.3]{BY}.
\end{remark}

The equivalence in Lemma~\ref{lem:BY}~(1) shows that in general $\butt(\pp)$ is not equivalent to $\butt(\salg)$. However, we have the following correspondence between cotorsion pairs in these two categories.

%Inspired by the above remark, we shall prove the following.

\begin{proposition}\label{prop:red}
	There is a bijection 
	$$(\m,\n)\mapsto(\Psi(\m),\Psi(\n))$$
	between cotorsion pairs in $\butt(\pp)$ and cotorsion pairs in $\butt(\salg)$, which restricts to a bijection between complete cotorsion pairs in $\butt(\pp)$ and complete cotorsion pairs in $\butt(\salg)$.
\end{proposition}

\begin{proof}
    Since by Lemma~\ref{lem:BY}~(1), $\Psi$ induces a bijection between isoclasses of objects of $\butt(\pp)$ and those of $\butt(\salg)$, so by Lemma~\ref{lem:BY}~(2), a pair $(\m,\n)$ of subcategories of $\butt(\s)$ is a cotorsion pair if and ony if $(\Psi(\m),\Psi(\n))$ is a cotorsion pair in $\butt(\salg)$. This gives the first required bijection. 
    
    Since $\Psi$ is a triangle functor, if $(\m,\n)$ is a complete cotorsion pair then so is its image $(\Psi(\m),\Psi(\n))$. By \cite[Proposition~A.3]{BY} (cf. also Remark~\ref{rmk:sil}), the map $\m\cap\n\mapsto\Psi(\m\cap\n)$ is a bijection between silting subcategories in $\butt(\pp)$ and those in $\butt(\salg)$. Using the bijection between silting subcategories and complete cotorsion pairs in Corollary~\ref{cor:sil-cotor}, we have that any complete cotorsion pair in $\butt(\salg)$ is of the form $(\Psi(\m),\Psi(\n))$ for a complete cotorsion pair $(\m,\n)$ in $\butt(\pp)$. This gives the second required bijection.

	%First note that if this is true, then this bijection is compatible with the bijection in %Remark~\ref{rmk:sil}.
\end{proof}

Now we are ready to show the main result in this section.

\begin{proof}[Proof of Theorem~\ref{thm:BB}]
    By Theorem~\ref{thm:hrs}, $(\vv(\pp),\uu(\pp)[1])$ is a complete cotorsion pair in $\butt(\pp)$. Then by Proposition~\ref{prop:red}, $(\m(\pp),\n(\pp))=(\Psi(\vv(\pp)),\Psi(\uu(\pp)[1]))$ is a complete cotorsion pair in $\butt(\salg)$. Thus, (c) holds.
    
    Since $$\m(\pp)\cap\n(\pp)=\Psi(\vv(\pp)\cap\uu(\pp)[1])=\Psi(\add\alg[1])=\add\qq,$$
    where the second equality is due to Theorem~\ref{thm:hrs}, by Corollary~\ref{cor:sil-cotor}, we have (d). 
    
    Finally, by Lemma~\ref{lem:BY}~(1), we have equivalences $\m(\pp)\simeq\vv(\pp)/(\pp[1],\pp)$ and $\n(\pp)\simeq\uu(\pp)[1]/(\pp[1],\pp)\simeq\uu(\pp)/(\pp,\pp[-1])$, which give (e).
\end{proof}

\begin{remark}
    The ideal $(\pp[1],\pp)$ of $\vv(\pp)$ has an intrinsic description: any morphism from $V_1$ to $V_2$ in $\vv(\pp)$ which is in $(\pp[1],\pp)$ is a composition $g\circ f$, where $f \colon V_1\to \pp'$ and $g \colon \pp'\to V_2$, with $\pp' \in\add\pp$ and such that $h\circ f=0$ for any $h \colon \pp''\to V_1$ with $\pp''\in\add\pp$.
\end{remark}

\section{Torsion pairs in the module category}

In this section we still work in the two-term category $\butt(\alg)$, and show how complete cotorsion pairs in the two-term category $\butt(\alg) = \alg \ast \alg[1]$, give rise to torsion pairs in the module category
$\mod \alg$. For any full subcategory $\vv$ of $\butt(\alg)$, we use $H_\alg^0(\vv)$ (or $H^0(\vv)$ for short if there is no confusion arising) to denote the full subcategory of $\mod\alg$ consisting of $H_\alg^0(V)$, $V\in\vv$.  The following result is well-known, cf. e.g. \cite{IYo}.

\begin{lemma}\label{lem:ideal}
The functor $H^0$ induces an equivalence 
$$\butt(\alg)/(\alg[1])\to\mod\alg,$$
where $(\alg[1])$ denotes the ideal of $\butt(\alg)$ consisting of the morphisms factoring through objects in $\add\alg[1]$.
\end{lemma}

%\begin{proof}
%    Note that $H^0=\Hom(\alg,-)$. So this lemma is well-known, see e.g. \cite{IYo}.
%\end{proof}

For the mains results here, Lemma \ref{lem:tor-cotor} and Proposition \ref{prop:tor-pairs}, there are more general statements and proofs in \cite{PZ}.

\begin{lemma}\label{lem:tor-cotor}
Let $(\uu,\vv)$ be a complete cotorsion pair in $\butt(\alg)$. Then we have $$H^0(\vv)=\Gen(H^0(\uu\cap\vv)).$$
\end{lemma}

\begin{proof}
    For any $V\in\vv$, since $(\uu,\vv)$ is a complete cotorsion pair, there is an $\et$-triangle
    $$V'\to U'\to V$$
    with $U'\in \uu$ and $V'\in\vv$. Since by Lemma~\ref{lem:ex}, $\vv$ is closed under extensions, we have $U'\in\uu\cap\vv$. Applying $H^0$ to the corresponding triangle
    $$V'\to U'\to V\to V'[1]$$
    in $\per\alg$, we have an exact sequence
    $$H^0(U')\to H^0(V)\to H^0(V'[1])=0$$
    in $\mod\alg$, which implies $V\in\Gen(H^0(\uu\cap\vv))$. So $H^0(\vv)\subset\Gen(H^0(\uu\cap\vv))$.
    
    For the converse inclusion, for any $M\in\Gen(H^0(\uu\cap\vv))$, by the equivalence in Lemma~\ref{lem:ideal}, there is a morphism $f:X\to Y$ in $\butt(\alg)$ such that $X\in\uu\cap\vv$, $H^0(Y)\cong M$ and $H^0(f)$ is an epimorphism in $\mod\alg$. Take a triangle
    \begin{equation}\label{eq:1}
       Z\to X\xrightarrow{f} Y\to Z[1] 
    \end{equation}
    in $\per\alg$ where $f$ sits. Then we have $Z\in\alg[-1]\ast\alg\ast\alg[1]$. Applying $H^0$ to the triangle~\eqref{eq:1}, we have an exact sequence
    $$H^0(X)\xrightarrow{H^0(f)} H^0(Y)\to H^0(Z[1])\to H^0(X[1])=0$$
    in $\mod\alg$. Since $H^0(f)$ is an epimorphism, we have $H^0(Z[1])=0$, which, together with $Z\in\alg[-1]\ast\alg\ast\alg[1]$, implies $Z\in\butt(\alg)$. Applying $\Hom_{\per\alg}(\uu,-)$ to the triangle~\eqref{eq:1}, we have an exact sequence
    $$\Hom_{\per\alg}(\uu,X[1])\to\Hom_{\per\alg}(\uu,Y[1])\to\Hom_{\per\alg}(\uu,Z[2])$$
    where the first term is zero by $X\in\uu\cap\vv$ and the last term is zero by $Z\in\butt(\alg)$. So we have $\et(\uu,Y)=0$, which implies $Y\in\vv$. So we have $M\cong H^0(Y)\in H^0(\vv)$. Hence $\Gen(H^0(\uu\cap\vv))\subset H^0(\vv)$.
\end{proof}

%\begin{lemma}\label{lem:factor}
%	Let $(\uu,\vv)$ be a complete cotorsion pair in $\butt(\alg)$. Then $H^0(\vv)$ is closed under factor objects.
%\end{lemma}

%\begin{proof}
%	Let $f \colon V\to X$ be a morphism in $\per \alg$ with $X\in\butt(\alg)$, $V\in\vv$ and such that $H^0(f)$ is an epimorphism in $\mod \alg$. Then in the triangle
%	$$V\xrightarrow{f} X\xrightarrow{g} Y\to V[1]$$
%	in $\per\alg$, we have that $H^0(g)=0$. By $X,V\in\alg\ast\alg[1] = \butt(\alg)$, this implies that $g$ factors through $\alg[1]$. Hence we have the following commutative diagram of triangles in $\per \alg$:
%	$$\xymatrix{
%		P_X\ar[r]\ar@{=}[d]&V\ar[r]\ar[d]^{f}&X'\ar@{-->}[d]\\
%		P_X\ar[r]&X\ar[r]\ar[d]^{g}&Q_X[1]\ar@{-->}[d]\\
%		&Y\ar@{=}[r]&Y
%	}$$
%	with $P_X,Q_X\in\add \alg$ and which produces a triangle $V\to X\oplus X'\to Q_X[1]\to V[1]$ in $\per \alg$. Note that $Q_X[1]\in\vv$ by Lemma~\ref{lem:projcon}. So $X\oplus X'\in\vv$ and hence $X$ is in $\vv$.
%\end{proof}

\begin{proposition}[{\cite[Theorem 3.6]{PZ}}]\label{prop:tor-pairs}
	The map
	$$(\uu,\vv)\mapsto (H^0(\vv),H^0(\vv)^\bot)$$
	is a bijection between complete cotorsion pairs in $\butt(\alg)$ and functorially finite torsion pairs in $\mod \alg$, with inverse
	given by $$(\T, \F) \mapsto (^{\bot_1}{\hat{\T}}, \hat{\T}) \text{ where }$$
	$$\hat{\T}=\{X\in\butt(\alg)\mid H^0(X)\in\T \} \text{ and }  ^{\bot_1}{\hat{\T}}= \{ Y \in \butt(\alg) \mid \Hom(Y, \hat{\T} [1]) = 0 \}.$$
\end{proposition}

\begin{proof}
	First recall that by \cite[Theorems 2.7 and 3.2]{AIR}, there is a bijection between silting subcategories in $\butt(\alg)$ and functorially finite torsion pairs in $\mod \alg$, given by mapping a silting complex
	$\pp$ to $(\Gen H^0(\pp) ,{(\Gen H^0(\pp))}^{\perp}) $. On the other hand, there is a bijection in Corollary~\ref{cor:sil-cotor} from the set of complete cotorsion pairs in $\butt(\alg)$ to the set of silting subcategories in $\butt(\alg)$, mapping $(\uu,\vv)$ to $\uu\cap\vv$. Then by Lemma~\ref{lem:tor-cotor}, we get the required bijection.
	
	%we only need to show $H^0(\vv)=\Gen H^0(\uu\cap\vv)$.
	
	%By Lemma~\ref{lem:factor}, we have the inclusion $\Gen H^0(\uu\cap\vv)\subset H^0(\vv)$.
	
	%Conversely, for any $V\in\vv$, by Corollary~\ref{cor:sil-cotor}, there is a triangle $S_1\to S_2\to V\to S_1[1]$ in $\per \alg$ with $S_1, S_2\in\uu\cap\vv$. We have   $H^1(S_1) =\Hom(\alg, S_1[1]) =0 $, since $S_1$ is in $\butt({\alg}) = \alg \ast \alg[1]$. Hence, taking homology of this triangle gives an epimorphism $H^0(S_2) \to H^0(V)$. This shows the reverse inclusion $\Gen H^0(\uu\cap\vv)\supset H^0(\vv)$, and this concludes the proof of the proposition.
\end{proof}

\begin{remark}\label{rmk:sil-tor-cotor}
    Let $(\uu,\vv)$ be a complete cotorsion pair in $\butt(\alg)$. By Theorem~\ref{thm:BB}, there is a silting complex $\pp$ in $\per\alg$ such that $\uu\cap\vv=\add\pp$. Hence, by Lemma~\ref{lem:tor-cotor} and \cite[Proposition~2.4]{BZ}, we have $$(H^0(\vv),H^0(\vv)^\bot)=(\T(\pp),\F(\pp)),$$
    where
	$$\T(\pp)=\{X\in\mod\alg\mid\Hom(\pp,X[1])=0\},\ \F(\pp)=\{X\in\mod\alg\mid\Hom(\pp,X)=0\}.$$

    %Since $H^0(\vv)=\Gen H^0(\uu\cap\vv)$ (Lemma~\ref{lem:tor-cotor}), the torsion pair $(H^0(\vv),H^0(\vv)^\perp)$ is induced by the silting subcategory $\s:=\uu\cap\vv$ in the sense that
	%$$(H^0(\vv),H^0(\vv)^\bot)=(\T(\s),\F(\s)),$$
	%where
\end{remark}

Let $\pp$ be a two-term silting complex in $\per\alg$ and let $\salg=\End\pp$. Recall from \cite{BZ} that there is a torsion pair $(\mathcal{X}(\pp),\mathcal{Y}(\pp))$ in $\mod\salg$, where
$$\mathcal{X}(\pp)=\Hom_{D^b(\alg)}(\pp,\F(\pp)[1]),\ \mathcal{Y}(\pp)=\Hom_{D^b(\alg)}(\pp,\F(\pp)[1]).$$
Moreover, there are equivalences $\mathcal{X}(\pp)\simeq \F(\pp)$ and $\mathcal{Y}(\pp)\simeq \T(\pp)$. On the other hand, by Theorem~\ref{thm:BB}, there is a complete cotorsion pair $(\uu(\pp),\vv(\pp))$ in $\butt(\alg)$ and a complete cotorsion pair $(\m(\pp),\n(\pp))$ in $\butt(\salg)$, with equivalences $\m(\pp)\simeq\uu(\pp)/(\pp[1],\pp)$ and $\n(\pp)\simeq\vv(\pp)/(\pp,\pp[-1])$. In the following, we show that $(\mathcal{X}(\pp),\mathcal{Y}(\pp))$ corresponds to $(\m(\pp),\n(\pp))$ under the bijection in Proposition~\ref{prop:tor-pairs} (but replacing $\alg$ with $\salg$). This implies that the Brenner-Butler theorem for torsion pairs is compatible with that for cotorsion pairs under the bijection in Proposition~\ref{prop:tor-pairs}.

\begin{proposition}\label{prop:compa}
We have the following equalities
$$H_{\salg}^0(\n(\pp))=\mathcal{X}(\pp),\ H_{\salg}^0(\n(\pp))^\bot=\mathcal{Y}(\pp).$$
\end{proposition}

\begin{proof}
    By \cite[Proposition~3.8]{BZ} and Theorem~\ref{thm:BB}~(d), both of the torsion pair $(\mathcal{X}(\pp),\mathcal{Y}(\pp))$ in $\mod\salg$ and the cotorsion pair $(\m(\pp),\n(\pp))$ in $\butt(\salg)$ are induced by the same two-term silting complex $\qq$. So by the proof of Proposition~\ref{prop:tor-pairs}, we get the required equalities.

%    By Proposition~\ref{prop:tor-pairs}, we only need to show the first equality, since any half part of (co)torsion pairs is determined by the other half part. By Lemma~\ref{lem:BY}, we have $\Phi(\vv[1])\cong \vv[1]/(S[1],S)$. So we have
%    $$H^0_\salg\left(\Phi(\vv[1])\right)\cong\Phi(\vv[1])/(\salg[1])\simeq\left(\vv[1]/(S[1],S)\right)/(S[1])=\vv[1]/(S[1])\simeq\vv/(S).$$
\end{proof}

%\begin{corollary}
%    $$H^0_\salg\left(\Psi(\uu[1])^\perp\right)\simeq H_\alg^0(\vv)\text{ and }H_\salg^0\left(\Psi(\uu[1])\right)\simeq H^0_\alg(\vv)^{\bot}.$$
%\end{corollary}

%\begin{verbatim}
%The end of this section needs attention. More details are needed for 3.3-5
%\end{verbatim}

\begin{corollary}\label{cor:4}
	Let $(\uu,\vv)$ be a complete cotorsion pair in $\butt(\alg)$. Then there is an equivalence $H^0(\vv)^\bot\simeq\uu/(\uu\cap\vv)$.
\end{corollary}

\begin{proof}
    By Corollary~\ref{cor:sil-cotor}, there is a two-term silting complex $\pp$ in $\per\alg$ such that $(\uu,\vv)=(\uu(\pp),\vv(\pp))$. Since $\mathcal{Y}(\pp)\simeq\T(\pp)$, by Remark~\ref{rmk:sil-tor-cotor} and Proposition~\ref{prop:compa}, we have $H^0(\vv)^\bot=\T(\pp)\simeq\mathcal{Y}(\pp)=H_{\salg}^0(\n(\pp))^\bot$. By definition, $\n(\pp)=\Psi(\uu[1])$, so we have the following equivalences $$H^0(\vv)^\bot\simeq H_\salg^0\left(\Psi(\uu[1])\right)\simeq \Psi(\uu[1])/(\salg[1])\simeq \left(\uu[1]/(\s[1],\s)\right)/\s[1]\simeq \uu[1]/\s[1]\simeq \uu/\s,$$
	where the second equivalence is due to Lemma~\ref{lem:ideal}, and the third one is due to Lemma~\ref{lem:BY}, 
\end{proof}

%\begin{remark}
%	We have $H^0(\vv)\simeq\vv/(\P[1])$.
%\end{remark}

\section{Weak cotorsion pairs and main results}

%	\begin{definition}
%		A pair $(\C,\T)$ of subcategories of $\mod \alg$ is called a left weak cotorsion pair (or lw-cotorsion pair for short) if
%		\begin{enumerate}
%			\item $\Ext^1(\C,\T)=0$;
%			\item for any $M\in\mod \alg$, there are exact sequences
%			$$0\to Y_M\to X_M\xrightarrow{f_M} M\to 0$$
%			and
%			$$M\xrightarrow{g^M} Y^M\to X^M\to 0$$
%			with $X_M,X^M\in\C$, $Y_M,Y^M\in\T$, $f_M$ a right $\C$-approximation of $M$, and $g^M$ a left $\T$-approximation of $M$. 
%		\end{enumerate}
%	\end{definition}

Cotorsion pairs for abelian categories were first studied by Salce \cite{S}, and by now
there are various definition of (complete) cotorsion pairs in the literature.
In our Definition \ref{def:lw-cot},
we note that requiring in addition $g_M$ to be injective, would
give what e.g. Hovey \cite{H} defines as a complete cotorsion pair.

First note that $\T$ determines $\C$ in a lw-cotorsion pair $(\C,\T)$, but the opposite
does in general not hold.
	
\begin{lemma}\label{c-determined}
For any left weak cotorsion pair $(\C,\T)$ of $\mod\alg$, we have $\C= {^{\bot_1}\T}$.	
\end{lemma}	

\begin{proof}
	We have by definition that $\C\subseteq  {^{\bot_1}\T}$.
	Let $M \in {^{\bot_1}\T}$. By definition there is an exact sequence 
	$$0\to Y_M\to X_M\to M\to 0$$ 
	with $X_M$ in $\C$ and $Y_M$ in $\T$. By assumption, the sequence splits, and 
	hence $M$ is in $\C$.
	\end{proof}
\begin{example}
Let $\alg =k Q/ (\beta \alpha)$, where $Q$ is the quiver 
$$1  \xrightarrow{\alpha} 2 \xrightarrow{\beta} 3.$$
Let $P_i$ and $S_i$ denote the indecomposable projective ans simple corresponding to vertex $i$ (such that $P_3= S_3$), let $\C = \add\{P_3, P_2, P_1, S_2\}$, let $\T= \add\{P_2, P_1, S_2\}$ and
let $\T'= \add\{P_2, P_1, S_2, S_1\}$. Then it is straightforward to check that both
$(\C,\T)$ and $(\C,\T')$ are lw-cotorsion pairs.
\end{example}	

Let $\tau$ denote the Auslander-Reiten translation in a module category.
Our aim is to see that support $\tau$-tilting modules
give rise to certain lw-cotorsion pairs. Recall that a basic
$\alg$-module $M$ is called $\tau$-tilting, if 
$\Hom(M, \tau M) = 0$, and $M$ has the same number of indecomposable summands as the number of simple $\alg$-modules,
and support $\tau$-tilting, if there is an idempotent $e$,
such that $M$ is $\tau$-tilting over $\alg/\alg e \alg$.

There is a close connection to two-term silting objects.

For a module $M$ in $\mod \alg$, we let $P_M$ denote its (punctured) minimal projective presentation, which we in natural way can view as a complex in $\per \alg$ concentrated
in degree $-1,0$, that is, as an object in $\butt(\alg)$.

\begin{lemma}[{\cite{AIR}}]\label{lem:air}
For $\alg$-modules $M,N$, we have that $\Hom(M, \tau N) = 0$ \sloppy if and only 
if $\Hom_{\per \alg}(P_N, P_M[1]) = 0$.
\end{lemma}

The complete cotorsion pairs in $\butt(\alg)$, give rise
to lw-cotorsion pair in $\mod \alg$.

\begin{proposition}\label{prop:left-weak}
Let $(\uu, \vv)$ be a complete cotorsion pair in $\butt(\alg)$. Then the pair
$(H^0(\uu), H^0(\vv) )$ is a lw-cotorsion pair in $\mod \alg$.
\end{proposition}	

\begin{proof}
We first show that $\Ext^1(X,Y)= 0$, when 
$X$ is in $H^0(\uu)$, and $Y$ is in $H^0(\vv)$.
Choose such $X,Y$, then $P_X \in \uu$ and $P_Y \in \vv$. 
Then, by assumption $\Hom_\t(P_X,P_Y[1]) = 0$, which implies $\Hom(Y, \tau X) = 0$ by Lemma \ref{lem:air}, and 
hence $\Ext^1(X,Y)= 0$ by the Auslander-Reiten formula.

Now let $M$ be any module in $\mod \alg$. Then, by assumption there are triangles
$$V \to U \xrightarrow{\alpha} P_M  \text{ \ and \ \ }  P_M \xrightarrow{\beta}  V' \to U' $$ 
in $\t$, with $U,U' \in \uu$ and $V,V' \in \vv$.

Applying $H^0$ to the first triangle gives an exact sequence 
$$H^0(V) \to H^0(U) \xrightarrow{H^0{\alpha}} M \to 0.$$
Then we have a short exact sequence
$$0\to K \to H^0(U) \xrightarrow{H^0{\alpha}} M \to 0$$
with $K$ a factor of $H^0(V)$. Since by Lemma~\ref{lem:tor-cotor} $H^0(\vv)=\Gen(H^0(\uu\cap\vv))$ is closed under factors, we have $K\in H^0(\vv)$. Since $\Ext^1(H^0(\uu), H^0(V))=0$, the map $H^0(\alpha):H^0(U) \to M$ is a right $H^0(\uu)$-approximation of $M$.

Consider now the right exact sequence obtained by applying $H^0$ to the second triangle
$$M \xrightarrow{H^0(\beta)} H^0(V') \to H^0(U') \to 0.$$
Let $N$ be in $H^0(\vv)$ and consider a map $g \colon M \to N$.
Choose a map $\gamma \colon P_M \to P_N$, such that $H^0(\gamma) = g$.
Since $P_N$ is in $\vv$, we have $\Hom_\t(U', P_N[1])= 0$, and hence there is a map
$\eta \colon V' \to P_N$ such that $\eta \beta = \gamma$ and hence $H^0(\eta) H^0(\beta) =g$, 
which shows that $H^0(\beta)$ is a left $H^0(\vv)$-approximation of $M$.
This finishes the proof that $(H^0(\uu), H^0(\vv) )$ is a lw-cotorsion pair.
\end{proof}

We are now ready to prove our first main theorem.

\begin{theorem}\label{lw-pair}
For a support $\tau$-tilting module $T$, we have that 
$({}^{\bot_1}\Gen T,\Gen T)$ is a lw-cotorsion pair.
\end{theorem}	

\begin{proof}
By \cite{AIR}, we have that $\Gen T$ is a functorially finite torsion class. 	
Then Proposition \ref{prop:tor-pairs}, implies there is some complete cotorsion pair 
$(\uu, \vv)$ in $\butt(\alg)$ such that $H^0(\vv) = \Gen T$.
Proposition \ref{prop:left-weak} now gives that $(H^0(\uu), H^0(\vv))$ is a lw-cotorsion pair,
and the statement hence follows from Lemma \ref{c-determined}. 
\end{proof}

We now prove the our second main theorem, which for finite dimensional algebras, is a generalization of \cite[Theorem 2.29]{BBOS}
by Bauer, Botnan, Oppermann and Steen (they work in a more general setting).
%Recall that a support $\tau$-tilting module is ....

\begin{theorem}\label{thm:2}
	The map $T \mapsto (^{\bot_1}{\Gen T}, \Gen T, T^{\bot})$ is a bijection between basic support $\tau$-tilting modules and
	lw-cotorsion-torsion triples, with inverse $(\C, \T, \F) \mapsto T$, where $\add T = \T \cap \F$.
	
	The map specializes to a bijection between tilting modules  and cotorsion-torsion triples.
\end{theorem}

\begin{proof}
	It follows from Theorems \ref{adachi} and \ref{lw-pair} that  $({}^{\bot_1}\Gen T,\Gen T,T^{\bot})$ is a lw-cotorsion-torsion triple, and
	it follows from \ref{adachi} that the map is injective.
	
	Let $(\C, \T, \F)$ be a lw-cotorsion-torsion triple.
	Then in particular $\T$ is a functorially finite torsion class and hence $\T= \Gen T$ for a support $\tau$-tilting module $T$. Surjectivity follows using again Theorem \ref{adachi}, in combination with Lemma \ref{c-determined}.
	
	The last part of the statement now follows from \cite[Theorem 2.29]{BBOS}, which says
	that the pair $ (^{\bot_1}{\Gen T}, \Gen T, T^{\bot})$ is a cotorsion-torsion triple, if $T$ is a tilting module in $\mod\alg$.
\end{proof}

\begin{corollary}
	For a lw-cotorsion-torsion triple $(\C, \T, \F)$, we have an equivalence $$\C/(\C\cap\T)\simeq\F.$$
\end{corollary}

\begin{proof}
    By Theorem~\ref{thm:2}, $\T=\Gen T$ for some support $\tau$-tilting module in $\mod\alg$. By \cite[Theorem~3.2]{AIR}, there is a silting complex $\pp$ in $\butt(\alg)$ such that $H^0(\pp)=T$. By \cite[Proposition~2.4]{BZ}, we have $\T=\T(\pp)$ and $\F=\F(\pp)$. Let $(\uu,\vv)=(\uu(\pp),\vv(\pp))$ as in Theorem~\ref{thm:BB}. Then by Remark~\ref{rmk:sil-tor-cotor}, we have $\T=H^0(\vv)$ and $\F=H^0(\vv)^\perp$. So by Proposition~\ref{prop:left-weak} and Lemma~\ref{c-determined}, we have $\C=H^0(\uu)$. Note that by Corollary~\ref{cor:4}, there is an equivalence $\uu/(\uu\cap\vv)\simeq H^0(\vv)^\perp$. This gives the required equivalence, because we have the following equivalences
	$$\begin{array}{rcl}
	     \C/(\C\cap\T)& \simeq& H^0(\uu)/(H^0(\uu)\cap H^0(\vv))\\
	     & \simeq &(\uu/(\alg[1]))/((\uu/(\alg[1]))\cap (\vv/(\alg[1])))\\
	     &=& (\uu/(\alg[1]))/((\uu\cap\vv)/(\alg[1]))\\
	     &\simeq &\uu/(\uu\cap\vv),
	\end{array}$$
	where the second equivalence is due to Lemma~\ref{lem:ideal}, and the last equivalence is due to the inclusion $(\alg[1])\subset(\uu\cap\vv)$ of ideals of $\uu$, which follows from that any morphism from $\uu$ to $\add A[1]$ factors through $\uu\cap\vv$ (see Remark~\ref{rmk:fac}).
\end{proof}

As a direct consequence of the above, we get the following.

\begin{corollary}
	For a support $\tau$-tilting module $T$, we have $$^{\bot_1}{\Gen T}/ \add T \simeq T^{\perp}.$$
\end{corollary}

\end{document}